\documentclass[12pt]{amsart} 
\usepackage[margin=1in]{geometry} 
\usepackage{amsmath,amsthm,amssymb}

\usepackage{latexsym}
\usepackage{amsmath, amscd, amsfonts}
\usepackage[square, comma, sort&compress, numbers]{natbib}

{\theoremstyle{plain}
    \newtheorem{Thm}{\bf Theorem}[section]
    \newtheorem{Prop}[Thm]{\bf Proposition}
    
    \newtheorem{Lma}[Thm]{\bf Lemma}
    \newtheorem{Cor}[Thm]{\bf Corollary}
    \newtheorem{Q}[Thm]{\bf Question}

}
{\theoremstyle{remark}
    \newtheorem{Rem}[Thm]{\bf Remark}
    \newtheorem{Exa}[Thm]{\bf Example}
}
{\theoremstyle{definition}
    \newtheorem{Def}[Thm]{\bf Definition}
}
\numberwithin{equation}{section}

\DeclareFontFamily{OMS}{rsfs}{\skewchar\font'60}
\DeclareFontShape{OMS}{rsfs}{m}{n}{<-5>rsfs5 <5-7>rsfs7 <7->rsfs10 }{}
\DeclareSymbolFont{rsfs}{OMS}{rsfs}{m}{n}
\DeclareSymbolFontAlphabet{\scr}{rsfs}

\newcommand{\mlabel}[1]%
  {\mbox{}\marginpar{\raggedleft\hspace{0pt}{\rm\ttfamily#1}}\label{#1}}

\newcommand{\excise}[1]{}
\title{\bf On the Frobenius Complexity of Stanley-Reisner Rings}
\author[I.~Ilioaea]{Irina Ilioaea}

\address{Department of Mathematics and Statistics, Georgia State University, Atlanta GA 30303}
\email{iilioaea1@gsu.edu}
\date{}

\begin{document}

\maketitle

\begin{abstract}
The Frobenius complexity of a local ring $R$ measures asymptotically the abundance of Frobenius operators of order $e$ on the injective hull of the residue field of $R$.
 It is known that, for Stanley-Reisner rings, the Frobenius complexity is either $-\infty$ or $0$. 
This invariant is determined by the complexity sequence $\{c_ e\}_e $ of the ring of Frobenius operators on the injective hull of the residue field. We will show that $\{c_ e\}_e $ is constant for $e\geq 2,$ generalizing work of \` Alvarez Montaner, Boix and Zarzuela. Our result settles an open question mentioned by \` Alvarez Montaner in~\cite{Alvarez}.

\end{abstract}

\section{Introduction}
Let $(R,m,k)$ be a Noetherian, local ring of positive characteristic $p,$ where $p$ is prime.
 The study of Frobenius operators on $R-$modules has played an important role in commutative algebra in the past several decades with applications to tight closure theory and singularities of local rings. In particular, Lyubeznik and Smith studied the algebra of Frobenius operators on the injective hull of the residue field of $R$ in connection with the localization problem in tight closure theory. They asked the following question in~\cite{LS}:
\begin{Q} [Lyubeznik, Smith] 
Is $\mathcal{F}(E_R)$ always finitely generated as a ring over $R$?
\end{Q}
In 2009, in~\cite{Kat}, Katzman gave an example of a Stanley-Reisner ring $R$ such that $\mathcal{F}(E_R)$ is not finitely generated as a ring over $R$. Subsequently, \` Alvarez Montaner, Boix and Zarzuela have completely described the situation in the case of Stanley-Reisner rings in~\cite{Boix}:


\begin{Thm} [\`Alvarez Montaner, Boix and Zarzuela] 
The Frobenius algebra $\mathcal{F}(E_R)$ associated to a Stanley-Reisner ring $R$ is either principally generated or infinitely generated.
\end{Thm}

In both papers, ~\cite{Boix} and ~\cite{Kat}, the infinite generation of the Frobenius algebra has been established by exhibiting, for $e>>0,$ a Frobenius operator of order $e,$ with the property that it is not part of the subring generated by the Frobenius operators of order strictly less than $e.$ In this paper, I will describe all operators with this property, for all $e.$

Enescu and Yao were motivated by the finite generation question to introduce a new invariant of a local ring of prime characteristic, called the Frobenius complexity in~\cite{EY}. This gives a systematic
way of measuring the abundance of Frobenius actions on the injective hull of the residue field of a local ring. In the case when this Frobenius algebra is infinitely generated over $R$, they used the complexity sequence in order to describe how far it is from being finitely generated. The complexity sequence counts the number of algebra generators in each degree which cannot be obtained from elements of lower degree.
From this perspective, the main result of this note provides the full description of the complexity sequence of the Frobenius algebra of any Stanley-Reisner ring.
This settles an open question mentioned by \`Alvarez Montaner in~\cite{Alvarez}, page $318.$




Let us introduce the main definitions and terminology relevant to the paper.
From now on, let $k$ be a field of characteristic $p$ and $(R, m, k)$ a complete local ring.

Let $F^e : R \to R$ be the $e$th iteration of the Frobenius map, that is $F^e(r)=r^q$, where $q=p^e, e\in \mathbb{N}$.

For any $e\geq 0$, we let $R^{(e)}$ be the $R$-algebra defined as follows: as a ring $R^{(e)}$ equals $R$ while the $R$-algebra structure is defined by $rs = r^q s$, for all $r\in R$, $s\in R^{(e)}.$ In the same way, starting with an $R$-module M, we can define a new $R$-module $M^{(e)}.$ 

Let $E_R:=E_R(k)$ denote the injective hull of the residue field $k$.
An $e$th Frobenius operator(or $e$th Frobenius action) of $E_R$ is an additive map $\phi: E_R \to E_R$ such that $\phi(rz)=r^q \phi(z)$, for all $r\in R$ and $z\in E_R$. The collection of $e$th Frobenius actions on $E_R$ is an $R$-module, denoted by $\mathcal{F}^e(E_R)$.

\begin{Def}
The algebra of the Frobenius operators on $E_R$ is defined by
\begin{center}
	$\mathcal{F}(E_R)=\oplus_{e\geq 0}\mathcal{F}^e(E_R)$.
\end{center} 
\end{Def}
This is a $\mathbb{N}$-graded noncommutative ring under composition of maps and due to Matlis duality, its zero degree component is $R$.

Let $G_e:=G_e(\mathcal{F}(E_R))$ be the subring of $\mathcal{F}(E_R)$ generated by elements of degree less or equal to $e$.
Note that $G_{e-1}\subseteq G_{e}$, for all $e\geq 1.$
Moreover, $(G_e)_i = \mathcal{F}(E_R)_i$, for all $0\leq i\leq e$ and $(G_{e-1})_{e}\subseteq \mathcal{F}(E_R)_{e}.$

We will denote the minimal number of homogeneous generators of $G_e$ as a subring of $\mathcal{F}(E_R)$ over $\mathcal{F}(E_R)_0=R$ by $k_e.$ The sequence $\{k_e\}_e$ is called the {\it growth} sequence for $\mathcal{F}(E_R)$. By convention, $k_{-1}=0.$ Note that $k_0=0.$

\begin{Prop}[Enescu, Yao]

The minimal number of generators of the $R$-module $\displaystyle\frac{\mathcal{F}(E_R)_{e}}{(G_{e-1})_{e}}$ equals $k_{e}-k_{e-1},$ for all $e\geq 1$.
\end{Prop}

The following definitions are central to our paper.

\begin{Def}\label{complexity}
\begin{itemize}
\item[{(i)}]  The {\it complexity sequence} of $\mathcal{F}(E_R)$ is given by $\{c_{e}= k_{e}-k_{e-1}\}_e$. Note that $c_0=0.$
The {\it complexity} of $\mathcal{F}(E_R)$ is 
\begin{center}
	$cx(\mathcal{F}(E_R)) = inf \{ n>0: c_{e} = O(n^e)\}.$
\end{center}
It is obvious that $cx(\mathcal{F}(E_R))=0$ when $\mathcal{F}(E_R)$ is finitely generated as a ring over $R.$

\item[{(ii)}] The {\it Frobenius complexity} of the ring $R$ is defined by 
\begin{center}
	$cx_F(R) = log_p(cx(\mathcal{F}(E_R))).$
	
\end{center}
\end{itemize}
\end{Def}
Under the notations above, Theorem~\ref{main} of this note shows that the complexity sequence $\{c_{e}\}_e$ of the Frobenius algebra of operators on the injective hull of the residue field of any Stanley-Reisner ring stabilizes starting with $e=2.$
\section{Preliminary Results}

In order to investigate the finite generation of the Frobenius algebra of operators on the injective hull, we will use a consequence of a result by Fedder in~\cite{Fed}. This gives us a nice description of the Frobenius algebra for a quotient of a regular ring.

\begin{Prop} [Fedder]\label{Fedder}
Let $k$ be a field of characteristic $p$, $S=k[[x_1,\ldots,x_n]]$ and $q=p^e$, for $e\geq 0$.
Let $I\leq S$ be an ideal in $S$ and $R=S/I$.

There exists an isomorphism of $R$-modules:
\begin{center}
	$\mathcal{F}^e(E_R)\cong \displaystyle\frac{I^{[q]}:_S I}{I^{[q]}}.$
\end{center}
Therefore, 
\begin{center}
$\mathcal{F}(E_R)\cong \displaystyle\bigoplus_{e\geq 0}\displaystyle\frac{I^{[q]}:_S I}{I^{[q]}}.$
\end{center}
\end{Prop}
In~\cite{Kat}, Katzman described the $e$th Frobenius actions that come from Frobenius actions of lower degree $e',$ with $e'<e$.
Using the notations in Proposition~\ref{Fedder}, for any $e\geq 0$ denote $K_e:= (I^{[p^e]}:_S I)$ and 
\begin{center}
	$L_e := \displaystyle\sum_{1\leq \beta_1,\ldots,\beta_s< e, 
		\beta_1+\ldots+\beta_s=e} K_{\beta_1} K_{\beta_2}^{[p^{\beta_1}]}\cdots K_{\beta_s}^{[p^{\beta_1+\cdots+\beta_{s-1}}]}.$
\end{center}
\begin{Prop} [Katzman]  
For any $e\geq 1$, let $\mathcal{F}_{<e}$ be the $R$-subalgebra of $\mathcal{F}(E_R)$ generated by $\mathcal{F}^0(E_R),\ldots,\mathcal{F}^{e-1}(E_R).$ Then 
\begin{center}
	$\mathcal{F}_{<e}\cap\mathcal{F}^e(E_R) = L_e .$
\end{center}
\end{Prop}
Therefore, $(G_{e-1})_e\cong\displaystyle\frac{L_e+I^{[q]}}{I^{[q]}} $ and  $c_e=\mu_S\left(\displaystyle\frac{I^{[q]}:_S I}{L_e+I^{[q]}}\right).$ 

\vskip 0.4cm
Let $\textbf{x}^1$ denote the product of all the variables, i.e.
$\textbf{x}^1=x_1\cdots x_n.$
\begin{Def}
We define $J_q$ to be the unique minimal monomial ideal satisfying the equality
\begin{center}
	$(I^{[q]}: I) = I^{[q]} + J_q + (\textbf{x}^1)^{q-1}.$
\end{center}

\end{Def}

I would like to present the structure of the ideal $J_q$ for a Stanley-Reisner ring such that the simplicial complex associated to it has no isolated vertices. For the remaining part of this section, we are working in this context.

Let $k$ be a field of characteristic $p$, $S=k[[x_1,\ldots,x_n]]$ and $q=p^e$, for $e\geq 0$.
Let $I\leq S$ be a square-free monomial ideal in $S$ and $R=S/I$ the Stanley-Reisner ring associated to $I$.

Let $\alpha_k=(\alpha_{k1},\ldots,\alpha_{kn})\in \{0,1\}^n$, $1\leq k\leq r$, be distinct vectors. The support of the vector $\alpha_k$ is defined as $ supp(\alpha_k)=\{i: \alpha_{ki}=1\}.$

We consider $I_{\alpha_k}=(x_i: i\in supp(\alpha_k))$, for every $1\leq k\leq r$ and $x^{\alpha_k}=x_1^{\alpha_{k1}}\cdots x_n^{\alpha_{kn}}$ such that $I_{\alpha_1}+ I_{\alpha_2}+\cdots+ I_{\alpha_r}=(x_1,\ldots, x_n).$

In~\cite{Boix}, \` Alvarez Montaner, Boix and Zarzuela presented a formula to compute the colon ideals $(I^{[q]}: I)$ based on the minimal primary decomposition of the ideal $I$.
\begin{Prop} [\` Alvarez Montaner, Boix and Zarzuela]
If $I = I_{\alpha_1}\cap I_{\alpha_2}\cap\ldots\cap I_{\alpha_r}$ is the minimal primary decomposition of the ideal $I$, then
$$(I^{[q]}:_S I)= (I_{\alpha_1}^{[q]}:_S I_{\alpha_1})\cap\ldots\cap (I_{\alpha_r}^{[q]}:_S I_{\alpha_r}) = (I_{\alpha_1}^{[q]}+(x^{\alpha_1})^{q-1})\cap\cdots\cap(I_{\alpha_r}^{[q]
}+(x^{\alpha_r})^{q-1}).$$

\end{Prop}

\begin{Rem}
\label{lcm}
\begin{itemize}
\item[{(i)}] Since the ideals in the intersection are monomial ideals, one can compute the minimal monomial generators of the ideal $(I^{[q]}: I)$ by taking the least common multiples of the minimal monomial generators of the ideals $(I_{\alpha_i}^{[q]}+(x^{\alpha_i})^{q-1}).$

In this way, we can see that the minimal generators $x^{\gamma}=x_1^{\gamma_1}\cdots x_n^{\gamma_n}$ of $(I^{[q]}: I)$ satisfy $\gamma_i\in\{0,q-1,q\}$.

\item[{(ii)}] One can notice that the formula obtained for $(I^{[q]}: I)$ depends only on $q$ and on the vectors $\alpha_i$'s. Since the vectors $\alpha_i$ are invariants of the ideal $I$, we can obtain the minimal monomial generators of $(I^{[q]}: I)$ from the minimal monomial generators of $(I^{[p]}: I)$ by changing $p$ into $q.$ 
\end{itemize}
\end{Rem}

\begin{Exa}\label{ex}
Let $I=(x_1x_5, x_2x_5, x_2x_3, x_2x_4)$.
Then 

	$$(I^{[q]}: I) = ( x_1^qx_5^q , x_2^qx_5^q , x_2^qx_3^q , x_2^qx_4^q , x_1^{q-1}x_2^{q-1}x_5^q ,  x_2^qx_3^{q-1}x_4^{q-1}x_5^{q-1} ,$$

	$$x_1^{q-1}x_2^{q-1}x_4^qx_5^{q-1} , x_1^{q-1}x_2^{q-1}x_3^qx_5^{q-1} , x_1^{q-1}x_2^{q-1}x_3^{q-1}x_4^{q-1}x_5^{q-1})$$
	
and therefore
\begin{center}
	$J_q = ( x_1^{q-1}x_2^{q-1}x_5^q,  x_2^qx_3^{q-1}x_4^{q-1}x_5^{q-1}, x_1^{q-1}x_2^{q-1}x_4^qx_5^{q-1}, x_1^{q-1}x_2^{q-1}x_3^qx_5^{q-1}).$
\end{center}

\end{Exa}
\vskip 0.5cm
\begin{Lma}\label{q}
We have that $J_q\neq 0$ if and only if there exists a generator $x^{\gamma}\in (I^{[q]}: I)$ having $\gamma_i=q$, $\gamma_j=q-1$ and $\gamma_k=0$ for some $1\leq i,j,k\leq n$.
\end{Lma}
\begin{proof}
It is trivial to see that if there exists $x^{\gamma}\in (I^{[q]}: I) $ with $\gamma_i=q$, $\gamma_j=q-1$ and $\gamma_k=0$ for some $1\leq i,j,k\leq n$, then $J_q\neq 0.$

Let us assume that $J_q\neq 0$. By Remark~\ref{lcm}, if $x^{\gamma}\in J_q$, then $x^{\gamma}=x_1^{\gamma_1}\cdots x_n^{\gamma_n}$ must have $\gamma_i\in\{0,q-1,q\},$ for all $i\in\{1,\ldots,n\}$. 

Moreover, $x^{\gamma}=lcm(x^{\theta_1},\ldots,x^{\theta_r})$, where $x^{\theta_i}\in(I_{\alpha_i}^{[q]}+(x^{\alpha_i})^{q-1}),$ for $i\in\{1,\ldots,r\}$.

If $\gamma_i\neq 0,$ for all $i\in\{1,\ldots,n\}$, then $(\textbf{x}^1)^{q-1}$ divides $x^{\gamma},$ hence $x^{\gamma}\in (\textbf{x}^1)^{q-1}.$

If $\gamma_i\neq q,$ for all $i\in\{1,\ldots,n\}$, then we must have  $x^{\theta_i}\in((x^{\alpha_i})^{q-1})$, for all $i\in\{1,\ldots,r\}$. But that happens only if $x^{\gamma}\in(\textbf{x}^1)^{q-1}.$

If $\gamma_i\neq q-1,$ for all $i\in\{1,\ldots,n\}$, then there exists at least one $x^{\theta_i}\in(I_{\alpha_i}^{[q]})$, hence we have that $x^{\gamma}\in I^{[q]}.$

Therefore, if $J_q\neq 0$, then there exists at least one generator $x^{\gamma}\in J_q$ with $\gamma_i=q$, $\gamma_j=q-1$ and $\gamma_k=0$ for some $1\leq i,j,k\leq n$.
\end{proof}
\vskip 0.5cm
In~\cite{Boix}, \` Alvarez Montaner, Boix and Zarzuela found that there are only four cases that may occur, considering the minimal primary decomposition of the ideal $I$ and the heights of the ideals $I_{\alpha_i}$ :
\begin{Prop}
There are only four posibilities for the minimal generators of $(I^{[q]}: I)$:
\begin{itemize}
  \item[{(i)}] Assume $ht(I_{\alpha_i}) > 1$, for all $i=1,\ldots,r.$
\begin{itemize}
  \item[{(a)}] $(I^{[q]}: I) = I^{[q]} + (\textbf{x}^1)^{q-1}.$
   \item[{(b)}] $(I^{[q]}: I) = I^{[q]} + J_q + (\textbf{x}^1)^{q-1}$, $J_q\subsetneq I^{[q]} + (\textbf{x}^1)^{q-1}. $

\end{itemize}
\item[{(ii)}] Assume $ht(I) = 1$ and and there exists an $i\in \{1,\ldots,r\}$ such that $ht(I_{\alpha_i}) > 1.$

In this case, $(I^{[q]}: I) = J_q + (\textbf{x}^1)^{q-1}$, with $J_q\subsetneq (\textbf{x}^1)^{q-1}. $
\item[{(iii)}] Assume $ht(I_{\alpha_i}) = 1$ for all $i\in \{1,\ldots,r\}.$

Then $(I^{[q]}: I) = (\textbf{x}^1)^{q-1}.$
\end{itemize}
The Frobenius algebra $\mathcal{F}(E_R)$ is principally generated in cases $(i.a)$ and $(iii)$ and is infinitely generated in cases $(i.b)$ and $(ii)$.
\end{Prop}
\vskip 0.5cm
\begin{Rem}
\begin{itemize}
\item[{(i)}] In the case when $\mathcal{F}(E_R)$ is principally generated, it is generated by $(\textbf{x}^1)^{p-1}$. 
When $\mathcal{F}(E_R)$ is infinitely generated, $\mathcal{F}^1(E_R)$ has $\mu +1$ minimal generators, $\mu$
of them being the minimal generators of $J_p$ and $(\textbf{x}^1)^{p-1}.$
Each graded piece $\mathcal{F}^e(E_R)$ adds up $\mu$ new generators coming from $J_q$.

\item[{(ii)}] The complexity sequence $\{c_e\}_{e\geq 2}$ is bounded by above by the minimal number of generators of the ideal $J_p$, i.e. $c_e\leq \mu_S(J_p),$ for any $e\geq 2.$ Note that $c_1=\mu +1$ and $c_0=1.$
\end{itemize}
\end{Rem}

\begin{Def}
Let $Supp(J_q)$ be the set of all the supports of the minimal monomial generators of $J_q.$
We define $\Gamma : = Supp(J_q)$ to be the \textit{support set} of the ring $R$. 
Then $(\Gamma, \subseteq)$ is a partially ordered set.
\end{Def}
\begin{Def}
Let $\Gamma$ be the support set of a Stanley-Reisner ring $R=S/I.$ Let $Min(\Gamma)$ be the set of elements in $\Gamma$ which are minimal with respect to inclusion.
We call $\Gamma$ \textit{minimal} if $\Gamma\neq\emptyset$ and $Min(\Gamma)=\Gamma$.

\end{Def}
\begin{Exa}
Let $I=(x_1x_2, x_1x_3, x_2x_4)$.
Then

\begin{center}
	$J_q = ( x_1^{q}x_2^{q-1}x_3^{q-1},  x_1^{q-1}x_2^{q}x_4^{q-1}).$
\end{center}
The support set is 
$$\Gamma=\{(1,2,3), (1,2,4) \}.$$
In this case, $\Gamma$ is minimal.
\end{Exa}

\begin{Exa}

The support set of the ideal in Example~\ref{ex} is 
$$\Gamma=\{(1,2,5), (2,3,4,5), (1,2,4,5), (1,2,3,5) \}.$$
In this case, $\Gamma$ is not minimal.
\end{Exa}

\section{Main Result}
In this section, we will prove that the complexity sequence $\{c_e\}_{e\geq 0}$ of the Frobenius algebra of operators of the injective hull of the residue field of any Stanley-Reisner ring with non-empty support set stabilizes starting with $e=2$. 

Let $k$ be a field of characteristic $p$, $S=k[[x_1,\ldots,x_n]]$ and $q=p^e$, for $e\geq 0$.
Let $I\leq S$ be a square-free monomial ideal in $S$ and $R=S/I$ the Stanley-Reisner ring associated to $I$. We will assume that the simplicial complex associated to the ring $R$ has no isolated vertices and use the notations introduced in the previous section.

\begin{Lma}\label{claim}
Let $e\geq 0$ an integer and suppose that $J_q\neq 0.$ Let $x^{\gamma_e}$ be a minimal monomial generator of $J_q.$ If there exists a minimal monomial generator $x^{\gamma'_e}$ of $J_q$ with $supp(\gamma'_e)\subsetneq supp(\gamma_e)$, then there exists at least one variable $x_k$ such that
\begin{center}
$deg_{x_k}(x^{\gamma_e})=q-1$ and $deg_{x_k}(x^{\gamma'_e})=q.$
\end{center}
\end{Lma}
\begin{proof}
We will prove the lemma by contradiction. 
Assume not.
Then, for all the variables $x_k$ with $deg_{x_k}(x^{\gamma'_e})=q,$ we have that $deg_{x_k}(x^{\gamma_e})\neq q-1$, therefore $deg_{x_k}(x^{\gamma_e})\in\{0,q\},$ by Remark~\ref{lcm}.

But since $supp(\gamma'_e)\subsetneq supp(\gamma_e)$, we have that $deg_{x_k}(x^{\gamma_e}) > 0.$ 
Hence, $deg_{x_k}(x^{\gamma_e})= q.$

Then, we have that $x^{\gamma'_e}$ divides $x^{\gamma_e}$, which is a contradiction.
\end{proof}
\begin{Def}
Let $e\geq 0$ an integer and suppose that $J_p\neq 0.$
Let $x^{\gamma}\in J_p$ a minimal monomial generator. Using Remark~\ref{lcm}, we have a bijective correspondence between the minimal monomial generators of $J_p$ and the minimal monomial generators of $J_q$. Under this map, there exists $x^{\gamma_e}\in J_q$ which corresponds to $x^{\gamma}\in J_p.$ We define  $M_e(\gamma)\subseteq K_e$ to be the ideal generated by the minimal monomial generators $x^{\delta}\in J_q$ with $supp(x^{\delta})\subseteq supp(x^{\gamma_e})=supp(x^{\gamma})$.
\end{Def}
\begin{Lma}\label{ge}
Let $e\geq 0$ an integer and suppose that $J_p\neq 0.$
Let $x^{\gamma}\in J_p$ a minimal monomial generator. 
Let $1\leq \beta_1,\ldots,\beta_s< e$ with $\beta_1+\ldots+\beta_s=e$. 

Then, 
\begin{center}
$x^{\gamma_e}\in K_{\beta_1} K_{\beta_2}^{[p^{\beta_1}]}\cdots K_{\beta_s}^{[p^{\beta_1+\cdots+\beta_{s-1}}]}$
if and only if $x^{\gamma_e}\in M_{\beta_1}(\gamma) (M_{\beta_2}(\gamma))^{[p^{\beta_1}]}\cdots (M_{\beta_s}(\gamma))^{[p^{\beta_1+\cdots+\beta_{s-1}}]}.$
\end{center}
\end{Lma}
\begin{proof}
Let $x^{\gamma_e}\in M_{\beta_1}(\gamma) (M_{\beta_2}(\gamma))^{[p^{\beta_1}]}\cdots (M_{\beta_s}(\gamma))^{[p^{\beta_1+\cdots+\beta_{s-1}}]}.$

Since  $M_{\beta_i}(\gamma)\subseteq K_{\beta_i}$, we obtain that $x^{\gamma_e}\in K_{\beta_1} K_{\beta_2}^{[p^{\beta_1}]}\cdots K_{\beta_s}^{[p^{\beta_1+\cdots+\beta_{s-1}}]}.$

Now, let $x^{\gamma_e}\in K_{\beta_1} K_{\beta_2}^{[p^{\beta_1}]}\cdots K_{\beta_s}^{[p^{\beta_1+\cdots+\beta_{s-1}}]}.$ Then for every $i\in\{1,\ldots,s\},$ there exists $m_{\beta_i}\in K_{\beta_i}$ such that $x^{\gamma_e}= m_{\beta_1} m_{\beta_2}^{p^{\beta_1}}\cdots m_{\beta_s}^{p^{\beta_1+\cdots+\beta_{s-1}}}\cdot m,$ for some $m\in S.$ 

If there exists at least an $i$ with $supp(m_{\beta_i})\not\subseteq supp(x^{\gamma_e}),$ there exists at least one $x_k\in supp(m_{\beta_i})\setminus supp(x^{\gamma_e}).$ But this contradicts the equality $x^{\gamma_e}= m_{\beta_1} m_{\beta_2}^{p^{\beta_1}}\cdots m_{\beta_s}^{p^{\beta_1+\cdots+\beta_{s-1}}}\cdot m.$
Therefore, we must have that $supp(m_{\beta_i})\subseteq supp(x^{\gamma_e}),$ for all $i\in\{1,\ldots,s\}.$ 

Hence, $x^{\gamma_e}\in M_{\beta_1}(\gamma) (M_{\beta_2}(\gamma))^{[p^{\beta_1}]}\cdots (M_{\beta_s}(\gamma))^{[p^{\beta_1+\cdots+\beta_{s-1}}]}.$
\end{proof}
\begin{Prop}\label{stab}
Let $e\geq 2$ an integer and suppose that $J_q\neq 0.$ If all the minimal monomial generators of $J_q$ are not contained in $L_e$, then $c_e = c_{e+1}$, for all $e\geq 2$.
\end{Prop}
\begin{proof}
Let $e\geq 2$ and let $x^{\gamma_e}$ a minimal monomial generator of $J_q.$ We know that $x^{\gamma_e}$ is not contained in $L_e$. So, we obtain that $\bar{0}\neq\overline{x^{\gamma_e}}\in \left(\displaystyle\frac{I^{[q]}:_S I}{L_e+I^{[q]}}\right)$.
Since $x^{\gamma_e}$ was arbitrarly chosen in $J_q$ and $c_e=\mu_S\left(\displaystyle\frac{I^{[q]}:_S I}{L_e+I^{[q]}}\right)$, we have that $c_e= \mu_S(J_q),$ for all $e\geq 2$. Therefore, $c_e = c_{e+1}$, for all $e\geq 2.$
\end{proof}
\begin{Rem}
In order to show that the complexity sequence $\{c_e\}_{e\geq 0}$ stabilizes starting with $e=2$, it is enough to show that all the minimal monomial generators of $J_q$ are not contained in $L_e.$
\end{Rem}

\begin{Thm}
Let $e\geq 0$ an integer and suppose that $J_q\neq 0.$ Let $x^{\gamma_e}$ be a minimal monomial generator of $J_q.$ Then, $x^{\gamma_e}$ is not contained in $L_e$. 
\end{Thm}

\begin{proof}

Let $x^{\delta_e^{(i)}}\in J_q$ be a minimal monomial generator with 
\begin{center}
    $supp(\delta_e^{(1)}),\ldots, supp(\delta_e^{(k)})\subsetneq supp(\gamma_e),$

\end{center}
where $k\geq 0.$ Note that if $k=0,$ all the minimal monomial generators of $x^{\gamma_e}$ have minimal support. Note that
\begin{center}
    $M_e(\gamma):=(x^{\gamma_e},x^{\delta_e^{(j)}}: j=1,\ldots,k).$
\end{center}

We want to show that $x^{\gamma_e}\notin L_e.$

If $x^{\gamma_e}\in L_e$, there exists
$1\leq \beta_1,\ldots,\beta_s< e$ with $\beta_1+\ldots+\beta_s=e$ with $x^{\gamma_e}\in K_{\beta_1} K_{\beta_2}^{[p^{\beta_1}]}\cdots K_{\beta_s}^{[p^{\beta_1+\cdots+\beta_{s-1}}]}.$

Using Lemma~\ref{ge}, we have that $x^{\gamma_e}\in M_{\beta_1}(\gamma) (M_{\beta_2}(\gamma))^{[p^{\beta_1}]}\cdots (M_{\beta_s}(\gamma))^{[p^{\beta_1+\cdots+\beta_{s-1}}]},$ where

$M_{\beta_i}(\gamma):=(x^{\gamma_{\beta_i}},x^{\delta_{\beta_i}^{(j)}}: j=1,\ldots,k),$ for all $i\in\{1,\ldots,s\}.$ 

Then there exists $m_{\beta_i}\in M_{\beta_i}(\gamma)$
such that $m_{\beta_1}m_{\beta_2}^{p^{\beta_1}}\cdots m_{\beta_s}^{p^{\beta_1+\ldots +\beta_{s-1}}}$ divides $x^{\gamma_e}.$

By Lemma~\ref{claim}, we have the following:

For any $j\in\{1,\ldots,k\}$, there exists at least one variable $x_{o(j)}\in supp(\delta_e^{(j)})$ with 
\begin{center}
$deg_{x_{o(j)}}(x^{\gamma_e})=q-1$ and $deg_{x_{o(j)}}(x^{\delta_e^{(j)}})=q,$
\end{center}
for all $e\geq 0.$
If $m_{\beta_s}$ is a multiple of $x^{\delta_{\beta_s}^{(j)}}$, for some $j\in\{1,\ldots,k\},$ using the Lemma~\ref{claim} there exists 
$x_{o(j)}\in supp(\delta_e^{(j)})$ with 
\begin{center}
$deg_{x_o(j)}(x^{\gamma_{\beta_{s}}})=p^{\beta_s}-1$ and $deg_{x_o(j)}(x^{\delta_{\beta_{s}}^{(j)}})=p^{\beta_s}.$
\end{center}
Then, we obtain that 
\begin{center}
 $deg_{x_{o(j)}}(x^{\gamma_e})=q-1 \geq deg_{x_{o(j)}}(m_{\beta_s})\cdot p^{{\beta_1+\ldots +\beta_{s-1}}}\geq q,$   
\end{center}
which is a contradiction.

Therefore, we must have that $m_{\beta_s}$ is a multiple of $x^{\gamma_{\beta_s}}$.

Now for those variables $x_r$ with $deg_{x_{r}}(x^{\gamma_e})=q$, we have that 
\begin{center}
    $deg_{x_{r}}(x^{\gamma_e})=q \geq deg_{x_{r}}(m_{\beta_s})\cdot p^{{\beta_1+\ldots +\beta_{s-1}}} \geq q,$
\end{center}
so we must have equality.

That means that $deg_{x_{r}}(m_{\beta_1}m_{\beta_2}^{p^{\beta_1}}\cdots m_{\beta_{s-1}}^{p^{\beta_1+\ldots +\beta_{s-2}}})=0$, which implies that $m_{\beta_j}$ is not a multiple of $x^{\gamma_{\beta_j}},$ for all $j\in\{1,\ldots,s-1\}.$

In particular, we have that $m_{\beta_{s-1}}$ is a multiple of $x^{\delta_{\beta_{s-1}}^{(j)}},$ for some $j\in\{1,\ldots,k\}.$

Using the Lemma~\ref{claim} again, we know that there exists a variable $x_t\in supp(\delta_e^{(j)})$ with
\begin{center}
$deg_{x_t}(x^{\gamma_{\beta_{s-1}}})=p^{\beta_{s-1}}-1$ and $deg_{x_t}(x^{\delta_{\beta_{s-1}}^{(j)}})=p^{\beta_{s-1}}.$
\end{center}
Hence
\begin{center}
    $deg_{x_{t}}(x^{\gamma_e})=q-1 \geq deg_{x_{t}}(m_{\beta_{s-1}})\cdot p^{{\beta_1+\ldots +\beta_{s-2}}} + deg_{x_{t}}(m_{\beta_s})\cdot p^{{\beta_1+\ldots +\beta_{s-1}}}$
    \end{center}
    \begin{center}
   $ \geq p^{\beta_{s-1}}\cdot p^{{\beta_1+\ldots +\beta_{s-2}}} + (p^{\beta_{s}}-1)\cdot p^{\beta_1+\ldots +\beta_{s-1}}=q ,$

\end{center}
which gives us a contradiction.

We proved that $x^{\gamma_e}\notin L_e.$

\end{proof}
\begin{Cor}\label{stab}
Let $R$ be a Stanley-Reisner ring such that the simplicial complex associated to it has no isolated vertices. Then the complexity sequence of the Frobenius algebra of operators on the injective hull of the residue field of the ring $R$ is given by 
\begin{center}
	$\{c_e\}_{e\geq 0}=\{0,\mu +1,\mu,\mu,\mu,\ldots\},$ 
\end{center}
where $\mu : = \mu_S(J_p).$
\end{Cor}
\begin{Rem}
Corollary~\ref{stab} implies Theorem 4.9 in~\cite{BMZ}.
\end{Rem}

\subsection{General Case}
So far, we worked with Stanley-Reisner rings satisfying $I_{\alpha_1}+ I_{\alpha_2}+\cdots+ I_{\alpha_r}=(x_1,\ldots, x_n),$ and we showed that for these rings, the complexity sequence stabilizes starting with $e=2.$ 

Now our main goal will be to extend this result to all the Stanley-Reisner rings, by dropping the condition on the supports of the minimal prime ideals in the minimal primary decomposition of the ideal $I$.
\vskip 2mm

Let $c_{e,R}:= c_e(\mathcal{F}(E_R)).$
\begin{Thm}\label{flat}
Let $(S,m)\xrightarrow{}(T,n)$ be a flat, local extension of regular local rings and let $I\leq S$ be an ideal in $S.$ 
Let $R:=\displaystyle\frac{S}{I}$ and $R':=\displaystyle\frac{T}{IT}.$ Then, $c_{e,R}=c_{e,R'},$ for all $e\geq 0.$
\end{Thm}

\begin{proof}
We know that $c_{e,R}=\mu_S\left(\displaystyle\frac{I^{[q]}:_S I}{L_{e,R}+I^{[q]}}\right)$ and $c_{e,R'}=\mu_T\left(\displaystyle\frac{(IT)^{[q]}:_S (IT)}{L_{e,R'}+(IT)^{[q]}}\right),$ for all $e\geq 0.$  

Since $S\xrightarrow{}T$ is a flat extension of rings, we have that $(IT)^{[q]}:_S (IT)=(I^{[q]}:_S I)T,$  for all $e\geq 0.$ Hence, we obtain that $L_{e,R'}= L_eT.$
Moreover, $(IT)^{[q]}=I^{[q]}T,$ for all $e\geq 0.$ Therefore,$c_{e,R'}=\mu_T\left(\displaystyle\frac{(I^{[q]}:_S I)T}{(L_{e,R}+I^{[q]})T}\right),$ for all $e\geq 0.$ 

Let $A:= (I^{[q]}:_S I),$ $B:= (I^{[q]}+L_e)$ and $M:=\displaystyle\frac{A}{B}.$
Now in order to show that $c_{e,R}=c_{e,R'},$ for all $e\geq 0,$ it sufices to prove that $\mu_S(M)=\mu_T(T\otimes_S M).$

It is enought to show that $\mu_K\left(\displaystyle\frac{M}{mM}\right)=\mu_K\left(\displaystyle\frac{T\otimes_S M}{n(T\otimes_S M)}\right),$ where $K$ is the residue field $S/m.$

Let $K:=\displaystyle\frac{S}{m}$ and $L:=\displaystyle\frac{T}{n}$.
By tensoring the exact sequence 
\begin{center}
    $0\longrightarrow n\longrightarrow T\longrightarrow L\longrightarrow 0| \otimes_S M$
\end{center}
we obtain the following exact sequence 
\begin{center}
    $\ldots\longrightarrow Tor_1(L,M)\longrightarrow n\otimes_S M\longrightarrow T\otimes_S M\longrightarrow L\otimes_S M\longrightarrow 0.$
\end{center}
Hence, we have that $Im(n\otimes_S M\longrightarrow T\otimes_S M)=Ker(T\otimes_S M\longrightarrow L\otimes_S M).$

By the Fundamental Theorem of Isomorphism, 
\begin{center}
$\displaystyle\frac{T\otimes_S M}{Ker(T\otimes_S M\longrightarrow L\otimes_S M)}\cong Im(T\otimes_S M\longrightarrow L\otimes_S M).$

\end{center}

But the map $T\otimes_S M\longrightarrow L\otimes_S M$ is surjective, therefore 
$$Im(T\otimes_S M\longrightarrow L\otimes_S M)=L\otimes_S M.$$
Moreover, it is easy to see that $Im(n\otimes_S M\longrightarrow T\otimes_S M)= n(T\otimes_S M).$

Hence, we showed that 
\begin{center}
$\displaystyle\frac{T\otimes_S M}{n(T\otimes_S M)}\cong L\otimes_S M.$
\end{center}
In order to complete the proof, we will show that $\mu_L(L\otimes_S M)=\mu_K\left(\displaystyle\frac{M}{mM}\right).$

We have that 
$$L\otimes_S M\cong L\otimes_K\displaystyle\frac{S}{m}\otimes_S M\cong L\otimes_K \displaystyle\frac{M}{mM}. $$
Hence, $\mu_L(L\otimes_S M)=\mu_L(L\otimes_K \displaystyle\frac{M}{mM}).$ It is easy to see that $$\mu_L(L\otimes_K \displaystyle\frac{M}{mM})=\mu_K\left(\displaystyle\frac{M}{mM}\right),$$

which completes the proof.
\end{proof}
Corollary~\ref{stab} and Theorem~\ref{flat} give us the following result

\begin{Thm}\label{main}
Let $k$ be a field of characteristic $p$, $S=k[[x_1,\ldots,x_n]]$ and $q=p^e$, for $e\geq 0$.
Let $I\leq S$ be a square-free monomial ideal in $S$
and $R=S/I$ its Stanley-Reisner ring. 
Then,
\begin{center}
	$\{c_{e}\}_{e\geq 0}=\{0,\mu +1,\mu,\mu,\mu,\ldots\},$ 
\end{center}
where $\mu : = \mu_S(J_p).$
\end{Thm}
\begin{proof}
Let $I_{\alpha_1}+ I_{\alpha_2}+\cdots+ I_{\alpha_r}=(x_1,\ldots, x_m),$ where $1\leq m<n$. Since $k[[x_1,\ldots,x_m]]\subseteq k[[x_1,\ldots,x_n]]$ is a flat local extension of regular local rings, we are under the hypothesis of Theorem~\ref{flat}. Corollary~\ref{stab} combined with Theorem~\ref{flat} give us the desired conclusion.
\end{proof}
\begin{Rem}
Using the notations in the proof of~\ref{main}, one can notice that 
$\mu_{k[[x_1,\ldots,x_m]]}(J_p)=\mu_S(J_p).$
\end{Rem}

In~\cite{Alvarez}, \`Alvarez Montaner defined the generating function of a skew $R$-algebra using the complexity sequence. 

\begin{Def}{(see Definition 2.1 in~\cite{Alvarez})}
The generating function of $\mathcal{F}(E_R)$ is defined as 
\begin{center}
    $\mathcal{G}_{\mathcal{F}(E_R)}(T)=\displaystyle\sum_{e\geq 0}c_e T^e.$
\end{center}
\end{Def}
Note that in~\cite{Alvarez} the author takes $c_0=1.$

As a consequence of Theorem~\ref{main}, we obtain the generating function of the Frobenius algebra of operators on the injective hull of the residue field of any Stanley-Reisner ring.
\begin{Cor}
Let $k$ be a field of characteristic $p$, $S=k[[x_1,\ldots,x_n]]$ and $q=p^e$, for $e\geq 0$.
Let $I\leq S$ be a square-free monomial ideal in $S$, $R=S/I$ its Stanley-Reisner ring.

Then the generating function of the Frobenius algebra of operators is
\begin{center}
	$\mathcal{G}_{\mathcal{F}(E_R)}(T)=(\mu+1)T+\displaystyle\sum_{e\geq 2}\mu T^e =\displaystyle\frac{(\mu+1)T -T^2}{1-T}.$ 
\end{center}

\end{Cor}

\begin{proof}
Note that by Definition~\ref{complexity}, $c_0=0.$ Using the Theorem~\ref{main}, we have that $c_1(R)=\mu +1$ and $c_e=\mu$, for every $e\geq 2.$
\end{proof}
\bigskip
\section*{Acknowledgements} I would like to thank my advisor, Florian Enescu, for his constant support and guidance and Yongwei Yao for useful comments.

\end{document}